\documentclass{article}
\usepackage{amsmath,amssymb,amsfonts,amsthm,graphicx,latexsym}
\usepackage{srcltx}
\usepackage{tikz}

\newcommand\abs[1]{\left|#1\right|}
\newcommand\norm[2]{{\left\|#1\right\|}_{#2}}

\newcommand{\bu}{\bar u}
\newcommand{\yd}{y^\delta}
\newcommand{\df}{\mathcal{D}(F)}

\newcommand{\umn}{u_{m,n}^{\alpha,\delta}}

\newcommand{\R}{{\,\mathbb{R}\,}}
\newcommand{\N}{{\,\mathbb{N}\,}}
\newcommand{\Rcal}{\mathcal{R}}
\newcommand{\Ccal}{{\,\mathcal{C}\,}}
\newcommand{\dom}{\mathcal{D}}
\newcommand{\ufett}{\boldsymbol{u}}

\newcommand\set[1]{\left\{ #1 \right\}}

\newcommand{\signum}{\mathop{sign}}
\newcommand{\av}{{\boldsymbol{\alpha}}}
\newcommand{\bve}{{\boldsymbol{\beta}}}
\newcommand{\etafett}{\boldsymbol{\eta}}
\newcommand{\xifett}{\boldsymbol{\xi}}
\newcommand{\Deltaxih}{\Delta\left(\frac{\boldsymbol{x}-\xifett_\av}{h_n} \right)}

\include{tikzPics}

\newtheorem{theorem}{Theorem}[section]

\newtheorem{lemma}[theorem]{Lemma}
\newtheorem{example}[theorem]{Example}

\newtheorem{proposition}[theorem]{Proposition}
\newtheorem{remark}[theorem]{Remark}
\newtheorem{assumption}[theorem]{Assumption}

\begin{document}

\title{Discretization of  variational regularization in Banach spaces}

\author{Christiane P\"oschl
\footnote{Dept. of Information \& Communication Technologies
Universitat Pompeu Fabra
C/ T\'anger 122-140,
08018 Barcelona, Spain},\,
Elena Resmerita \footnote{Department of Industrial Mathematics, Johannes Kepler University, Altenbergerstra\ss{}e~69, A-4040 Linz, Austria,
\emph{elena.resmerita@jku.at}},\, and Otmar Scherzer
\footnote{Computational Science Center, University of Vienna,
Nordbergstra\ss{}e~15, A-1090 Vienna, Austria and Radon Institute of Computational and Applied Mathematics,
Altenbergerstra\ss{}e~69, A-4040 Linz, Austria \emph{otmar.scherzer@univie.ac.at}}}

\maketitle
\renewcommand{\thefootnote}{\fnsymbol{footnote}}
\renewcommand{\thefootnote}{\arabic{footnote}}

\begin{abstract}

Consider a nonlinear ill-posed operator equation $F(u)=y$ where $F$ is defined on a Banach space $X$.
In general, for solving this equation numerically, a finite dimensional approximation of $X$ and an
approximation of $F$ are required. Moreover, in general the given data $\yd$ of $y$ are noisy.
In this paper we analyze finite dimensional variational regularization, which takes into account operator
approximations and noisy data: We show (semi-)convergence of the regularized solution of the finite dimensional
problems and establish convergence rates in terms of Bregman distances under appropriate sourcewise representation
of a solution of the equation. The more involved case of regularization in nonseparable Banach spaces is
discussed in detail. In particular we consider the space of finite total variation functions, the space of functions
of finite bounded deformation, and the $L^\infty$--space.

\end{abstract}

\section{Introduction}\bigskip

Let $F:X\rightarrow Y$ be a   nonlinear operator with domain $\df$, where $X$  is a Banach space
and $Y$ is a Hilbert space. We would like to approximate solutions of the ill-posed equation
\begin{equation}\label{nonlin-eq}
F(u)=y\,
\end{equation}
via variational regularization.

Let $\Rcal: X \to [0,+\infty]$ be a  penalty functional with nonempty domain $\dom (\Rcal)$.
An element $\bar u \in {\cal{D}}(F)\cap \dom(\Rcal)$ is called an \textit{$\Rcal$-minimizing solution of}
(\ref{nonlin-eq}) if it solves the constraint optimization problem
\begin{equation}\label{solution}
\min \Rcal(u) \mbox{ subject to } F(u)=y.
\end{equation}
We assume that noisy data $\yd$ are given such that
\begin{equation}\label{noise}
\norm{\yd-y}{} \leq \delta.
\end{equation}

In order to solve equation (\ref{nonlin-eq}) numerically, the space
$X$ has to be approximated by a sequence of finite dimensional
subspaces $X_n$. The situation when the spaces $X$ and $Y$ are
Hilbert and the regularization is quadratic has been analyzed in
\cite{Neu89}, \cite{PlaVai90} (linear problems) and \cite{NeuSch90},\cite{Qin99} (nonlinear
problems). However, recent advances in regularization theory deal
with general Banach spaces. Although convergence of regularization
methods in the general setting has been established (see, e.g., \cite{ResSch06}. \cite{HofKalPoeSch07}, a unifying
discretization approach is still not at hand. In comparison with the
Hilbert space theory a significant complication is due to the fact
that a non-separable Banach space cannot be approximated by a nested
sequences of finite dimensional subspaces, with respect to the norm
topology. We have in mind the case of the space of bounded variation
functions BV or of bounded deformation functions BD, which are not
separable. However, as mentioned in \cite[page 121]{AmbFusPal00} in
the context of BV, ``the norm topology is too strong for many
applications'', and in particular also when considering finite
dimensional approximations.

In this work, we show how ill-posed nonlinear equations can be approximated by solving associated finite dimensional convex regularization problems.
The case of nonseparable Banach spaces is particularly emphasized. We propose to approximate such spaces $X$ by subspaces with respect to a
topology which is weaker than the norm topology. Instead of the norm topology on $X$ (as in the separable Hilbert space or in the separable Banach space setting) we
use a metric on $X$, which requires to have available an adequate superspace of $X$. Our investigation covers a large class of
not necessarily separable Banach spaces which are frequently used for regularizing imaging and other inverse problems.

The paper is organized as follows. Section \ref{analysis} specifies the assumptions, shows well-posedness and convergence of the discretized
regularization method in the case that the Banach space $X$ is nonseparable and the regularization term is a not necessarily convex function. Also, convergence
rates with respect to Bregman distances are obtained in the convex regularization setting, under
a standard source condition. The finite dimensional approximation of solutions of the equation in separable Banach spaces $X$ is briefly
discussed in Section \ref{separable}, following the analysis done in Section \ref{analysis}. Section \ref{partic_nonsep} studies in some
detail the discretization of several relevant nonseparable Banach spaces such as the space of bounded variation functions, the space of bounded
deformation functions and the space of essentially bounded functions. The inverse ground water filtration problem is analyzed in the natural setting of
$L^{\infty}(\Omega)$, in Section \ref{water}.

\section{The case of nonseparable Banach spaces}\label{analysis}

\subsection{Main assumptions}

Let $X$ be a not necessarily separable Banach space which can be
embedded into a separable Banach space $Z$.
Let $\tau$ be a topology on $X$ which is weaker than the norm topology on $X$.
Therefore, we refer to $\tau$ as the \emph{weak topology}.
In addition let $\Rcal:X \rightarrow [0,+\infty]$ be a proper functional.

We define a metric on the space $X$ by the norm of $Z$ induced on $X$ and a pseudometric generated by the function
$\Rcal$:
\begin{equation}\label{metric}
 d(u,v)=\|u-v\|_Z+|\Rcal(u)-\Rcal(v)|.
\end{equation}
We shall denote by $\tau_d$ the  topology generated by this metric. Relating to the above discussion, this is an intermediary topology
between the norm topology on $X$ and the weak topology $\tau$.

The following elementary result gives us a motivation for approximating a nonseparable Banach space in a  topology that is weaker than  the norm topology.

\begin{proposition} A Banach space $X$ is separable if and only if there exists a nested (increasing) sequence of finite dimensional subspaces $\{X_n\}$ such that

\[
{\overline{\cup_{n\in\mathbb{N}} X_n}}=X,
\]
where the closure is considered with respect to the norm topology of $X$.
\end{proposition}

Consider a sequence of nested, finite dimensional subspaces $\{X_n\}$  which satisfies
\begin{equation}\label{closure}
{\overline{\cup_{n\in\mathbb{N}} X_n}}^d=X;
\end{equation}
That is, $X$ is the closure of the reunion of the subspaces $X_n$
with respect to the topology of the metric $d$. This property holds for many  nonseparable Banach spaces - see several examples  in Section \ref{partic_nonsep}.

We are given approximation operators $F_m$ of $F$, which have the same domain $\df$ as
$F$. We assume that the operators $F,F_m$ satisfy the following approximation property:
\begin{equation}\label{closeness}
\|F(u)-F_m(u)\| \leq \rho_m \text{ for all } u\in\df\cap \dom(\Rcal)\;.
\end{equation}
Here, the constant $\rho_m$ should only depend on $m$ and satisfy
\begin{equation*}
\lim_{m\rightarrow\infty} \rho_m=0\;.
\end{equation*}
Denote
\[
D_n:=X_n\cap\df\cap \dom(\Rcal), n\in\mathbb{N},
\]
and assume that the sets $D_n$ are nonempty. We are interested in approximating
$\Rcal$-minimizing solutions of equation (\ref{nonlin-eq}) by
solutions $u_{m,n}^{\alpha,\delta}\in D_n$ of the problem
\begin{equation}\label{reg}
\min \left\{\norm{F_m(u)-\yd} { }^2+\alpha
\Rcal(u)\right\}\,\,\,\,\mbox{subject to} \,\,\,u\in D_n.
\end{equation}

In order to pursue the analysis, we make several (standard) assumptions on the spaces $X$, $Y$, the operator $F$,
the functional $\Rcal$ (see also (\cite{ResSch06,HofKalPoeSch07,SchGraGroHalLen08}), and on the approximations
$X_n$, $F_m$ as well:
\smallskip

\begin{assumption}
\label{newassump}

\begin{enumerate}
\item \label{it1} The Banach space $X$  is provided with  a topology
      $\tau$ such that
\begin{itemize}
\item The topology $\tau_d$ is finer than the topology $\tau$.
\item The norm topology is finer than the topology $\tau_d$.
\end{itemize}

\item \label{it2} The domain $\df$ is $\tau$-closed.
\item \label{it3} The operator $F: \df \subseteq X \to Y$ is sequentially $\tau$-weakly closed.
That is, $\{u_k\}\subset\df$, $u_k\stackrel{\tau}\to u$ and $F(u_k)\stackrel{w}\to v$ imply $u\in\df$ and $v=F(u)$. Moreover, the operator $F$
is continuous from $\df\subset Z$ with the norm topology to $(Y,\|\cdot\|)$
\item \label{it4} For every $m\in\mathbb{N}$, the operator $F_m$ is sequentially $\tau$-weakly closed.
\item \label{it5} The function $\Rcal$ is  sequentially
      $\tau$ - lower semi-continuous.
\item \label{it6} For every $M>0$, $\alpha>0$ and every $m,n\in\mathbb{N}$, the sets
      \begin{equation}\label{compact}
      \{u\in X_n: \norm{F (u)} { }^2+\alpha \Rcal(u)\leq M\}
      \end{equation}
      are $\tau$-sequentially relatively compact.
\item \label{it7} For every $u\in X$, there exists some
           $u_n \in X_n$ such that $d(u_n,u) \to 0$ as
      $n\rightarrow \infty$.
\end{enumerate}
\end{assumption}

\subsection{Well-posedness of the discretized problem}

We emphasize again one of the main ideas in this work: When discretizing a nonseparable Banach space, one could
work with topologies which are weaker than the original norm topology and which might be more natural than the
latter. A well-posedness result for problem \eqref{reg}  can be stated now
in this theoretical setting.\smallskip

\begin{proposition}
\label{pr3.1} Let $m,n\in\mathbb{N}$ and $\alpha,\delta>0$ be fixed.
Moreover, let assumptions \ref{newassump} and \eqref{closeness},
\eqref{noise} be satisfied.

Then, for every $y^{\delta} \in Y$, there exists at least one minimizer $u$ of \eqref{reg}.

Moreover, the minimizers of \eqref{reg} are stable with respect to the data $\yd$ in the
following sense:
if $\{y_k\}_{k\in\mathbb{N}}$ converges strongly to $\yd$, then every sequence
$\{u_k\}_{k\in\mathbb{N}}$ of minimizers of \eqref{reg} where $\yd$ is replaced by
$y_k$ has a subsequence $\{u_l\}_{l\in\mathbb{N}}$  which converges with respect to the topology
$\tau$ to a minimizer $\tilde u$ of \eqref{reg}  and such that $\{\Rcal(u_l)\}_{l\in\mathbb{N}}$
converges to $\Rcal(\tilde u)$, as $l\rightarrow\infty$.
\end{proposition}

The proof is analogous to the one for Theorem 3.23 in \cite{SchGraGroHalLen08},
which in turn generalizes a proof in \cite{EngKunNeu89} from
Hilbert to Banach spaces.

In the sequel, we prove (semi)convergence of the finite dimensional regularization method.

\begin{theorem}\label{th_conv} Let assumptions \ref{newassump} be satisfied. Moreover, assume that:

($i$) Equation (\ref{nonlin-eq}) has an $\Rcal$-minimizing solution $\bu$ in the interior of $\dom(\Rcal)\cap\df$, considered in the norm topology;

($ii$) $v_n\in \mathcal{D}(F)$ for $n$ sufficiently large, where $v_n\in X_n$ and $d(v_n,\bu) \to 0$ as
      $n\rightarrow \infty$;

($iii$) The parameter $\alpha=\alpha(m,n,\delta)$ is such that $\alpha\rightarrow 0$, $\frac{{\delta}^2}{\alpha}\rightarrow 0$, $\frac{{\rho_m}^2}{\alpha}\rightarrow 0$  and
\begin{equation}\label{cond}
\frac{\|F(v_n)-y\|}{\sqrt{\alpha}}\rightarrow 0
\end{equation}
as $\delta\rightarrow 0$ and $m,n\rightarrow\infty$.

If (\ref{closeness}), (\ref{noise}) also hold, then every sequence $\{u_k\}$, with $u_k:=u_{{m_k},{n_k}}^{\alpha_k,\delta_k}$ and $\alpha_k:=\alpha(m_k,n_k,\delta_k)$ where
$\delta_k\rightarrow 0$, $n_k\rightarrow\infty$, $m_k\rightarrow\infty$ as $k\rightarrow\infty$ and $u_k$ is a solution of (\ref{reg}),
has a  subsequence $\{u_l\}$ which converges with respect to the topology $\tau$ to an $\Rcal$-minimizing solution $\tilde u$ of equation (\ref{nonlin-eq}) and such that $\{\Rcal(u_l)\}_{l\in\mathbb{N}}$
converges to $\Rcal(\tilde u)$, as $l\rightarrow\infty$. Moreover, if $\bu$ is the unique solution of (\ref{nonlin-eq}), then the entire sequence $\{u_k\}$ converges to $\bu$ in the sense of $\tau$ and $\Rcal$ as above.
\end{theorem}
\medskip

\begin{proof}
From (\ref{metric}) and Assumption \ref{newassump}, Item \ref{it7}  it follows that $\Rcal(v_n)\rightarrow \Rcal(\bu)$
for $n\rightarrow\infty$. From the definition of $\umn$, the estimate (\ref{noise}) and (\ref{closeness}),
it follows that
\begin{eqnarray}\label{chain}
\|F_m(\umn)-\yd\| ^2+\alpha \Rcal(\umn)\leq \|F_m(v_n)-\yd\| ^2+\alpha \Rcal(v_n) \nonumber\\
\leq (\|F_m(v_n)-F(v_n)\| +\|F(v_n)-F(\bu)\| +\|F(\bu)-\yd\| )^2+\alpha \Rcal(v_n) \nonumber \\
\leq (\rho_m+\|F(v_n)-y\|+\delta)^2+\alpha \Rcal(v_n).
\end{eqnarray}
Therefore, 
 \[
 \Rcal(\umn)\leq \frac{(\rho_m+\|F(v_n)-y\|+\delta)^2}{\alpha}+\Rcal(v_n)\;.
 \]
Assumption ($iii$) now guarantees that
\begin{equation}\label{limsup}
\limsup \Rcal(\umn)\leq \Rcal(\bu).
\end{equation}
Observe that
 $\lim_{n\rightarrow\infty}\|F(v_n)-F(\bu)\|=0$
since $F$ is continuous at $\bu$ (compare Assumption \ref{newassump} Item \ref{it3}) and (i), (ii)).
By (iii), the quantity
$(\rho_m+\|F(v_n)-F(\bu)\|+\delta)^2$
in \eqref{chain}
converges to zero as $\delta\rightarrow 0$ and
$m,n\rightarrow\infty$, and then
\[
\lim\|F_m(\umn)-\yd\| =0\;.
\]
Therefore, by applying again (\ref{closeness}) and (\ref{noise}), we also have that
\begin{equation}\label{F_conv}
F(\umn)\rightarrow y,
\end{equation}
with respect to the norm of $Y$, as $\delta,\alpha\rightarrow 0$ and
$m,n\rightarrow\infty$.
 Consider $\alpha_k:=\alpha(m_k,n_k,\delta_k)$ and $u_k:=u_{{m_k},{n_k}}^{\alpha_k,\delta_k}$. Note that $\{\|F(u_k)\|^2+\alpha_k \Rcal(u_k)\}$
is bounded. Hence, by (\ref{closeness}) and by the compactness hypothesis in Assumption \ref{newassump}, there exists
a subsequence $\{u_j\}_{j\in\mathbb{N}}$  which is $\tau$-convergent to some $\tilde u\in D_n$.   Due to the lower semicontinuity of
$\Rcal$ and (\ref{limsup}), we get
\begin{equation*}
\Rcal(\tilde u)\leq \liminf_{j\rightarrow\infty}\Rcal(u_j)\leq \limsup_{j\rightarrow\infty}\Rcal(u_j)\leq \Rcal(\bu).
\end{equation*}
We also have that $F(\tilde u)=y$,
as $j\rightarrow\infty$ due to  (\ref{F_conv}) and Item \ref{it4} in Assumption \ref{newassump}.
Therefore, $\tilde u$ is an $\Rcal$ minimizing solution of equation (\ref{nonlin-eq}) and
\begin{equation*}
\Rcal(\tilde u)=\lim_{j\rightarrow\infty}\Rcal(u_j).
\end{equation*}
If the solution $\bu$ is unique, the only limit point of $\{u_k\}$ with respect to $\tau$ and $\Rcal$ is $\bu$.
\end{proof}
\begin{remark} In some situations, convergence of a sequence $\{u_k\}$ to $u$ with respect to the topology $\tau$ and such that
$\Rcal(u_k)\to\Rcal(u)$ implies $d(u_k,u)\to 0$ as $k\to\infty$. This happens for BV and BD which are embedded into
$Z=L^1$, and $\tau$ is chosen as the weak$^*$ topology, $\Rcal$ is the total variation and total deformation seminorm, respectively - see Section \ref{partic_nonsep}.
\end{remark}
\medskip

\subsection{Convergence Rates}

In this section, we  establish error estimates  for the
approximation method we analyze, with respect to Bregman distances. To this end, assume throughout this section that the regularization functional $\Rcal$ is convex.

Recall that \textit{the Bregman distance} with respect to a possibly non-smooth convex functional $\Rcal$  is defined by
\begin{equation*}
D_\Rcal(v,u)=\{D_{\Rcal}^\xi(v,u): \xi \in \partial \Rcal(u) \neq \emptyset\},
                               \quad u,v \in \dom(\Rcal),
\end{equation*}
where
\begin{equation*}
D_{\Rcal}^\xi(v,u)=\Rcal(v)-\Rcal(u)-\langle \xi,v-u\rangle.
\end{equation*}
More information about Bregman distances and their role in optimization and inverse problems can be found in
\cite{Res05}. Error estimates for variational or iterative regularization of (\ref{nonlin-eq}) by means of
a non-quadratic penalty have been shown in \cite{BurOsh04,Res05,ResSch06,HofKalPoeSch07,BurResHe07,FriSch09}.
The Bregman distance $D_{\Rcal}$ associated with $\Rcal$ was naturally chosen as the measure of discrepancy
between the error estimates.

 We assume Frechet differentiability of the operator $F$ around $\bu$
which is considered to be in the interior of $\mathcal{D}(F) \cap
\mathcal{D}(\Rcal)$; moreover, assume that its extension to the
space $Z$ is also Frechet differentiable around $\bu$. In fact, our
study is based on the following source-wise representation:

\emph{There exists $\omega \in Y$ such that}
\begin{equation}\label{bregsc}
\xi = F'(\bu)^*\omega \in \partial \Rcal(\bu),
\end{equation}
and  on the following nonlinearity condition (see also \cite{ResSch06}):

 \emph{There exist $\varepsilon>0$ and $c>0$ such that}
\begin{equation}\label{breg_condition}
\norm{F(u)-F(\bu)-F'(\bu)(u-\bu)} { } \leq c D_\Rcal(u,\bu),
\end{equation}
\emph{is satisfied for all $u\in\df\cap
U_{\varepsilon}(\bu)$ with respect to the above subgradient $\xi$
and such that}
\begin{equation}\label{ineq1}
c \norm{\omega} { }<1\;.
\end{equation}
 Here
\[
D_\Rcal(u,\bu)=\Rcal(u)-\Rcal(\bu)-\langle\xi,u-\bu\rangle,
\]
where  $\xi
\in\partial \Rcal(\bu)$ satisfies (\ref{bregsc}). Let us denote by
\begin{equation}\label{gamma}
\gamma_n:=\|F'(\bu)(v_n-\bu)\| ,
\end{equation}
\begin{equation}\label{lambda}
\lambda_n:=D_\Rcal(v_n,\bu).
\end{equation}
Here $\{v_n\}$ is a sequence as in Theorem \ref{th_conv}.

In the following we derive a relation between the Bregman distance
and the metric $d$ at $\bu$. Observe that
\[
D_\Rcal(v_n,\bu)=\Rcal(v_n)-\Rcal(\bu)-\langle\omega,F'(\bu)(v_n-\bu)\rangle.
\]
One has  $\Rcal(v_n)-\Rcal(\bu)\rightarrow 0$ since $d(v_n,\bu)\rightarrow 0$, as $n\rightarrow\infty$
(see (\ref{metric})).
Moreover, $F'(\bu)(v_n-\bu)\rightarrow 0$. As a consequence, $\lim_{n\rightarrow\infty}\lambda_n=0$.
Thus, convergence with respect to the metric $d$ is stronger than convergence with respect to the related Bregman distance.

\begin{theorem}
Suppose that Assumption \ref{newassump}, the assumptions in the
previous result, inequalities (\ref{noise}), (\ref{bregsc}) and
(\ref{breg_condition})  hold. Moreover, assume that
$\rho_m=O(\delta+\lambda_n+\gamma_n)$, with $\lambda_n$, $\gamma_n$
given by (\ref{lambda}), (\ref{gamma}). If $\alpha\sim
\max\{\delta,\lambda_n,\gamma_n\}$, then
\begin{equation}\label{error}
 D_\Rcal(\umn,\bu)=O(\delta+\lambda_n+\gamma_n).
\end{equation}
\end{theorem}\medskip

\begin{proof}
We have
\begin{eqnarray*}
\|F_m(\umn)-\yd\| ^2+\alpha \Rcal(\umn)\leq \|F_m(v_n)-\yd\| ^2+\alpha \Rcal(v_n)\\
\leq (\|F_m(v_n)-F(v_n)\| +\|F(v_n)-F(\bu)\| +\|F(\bu)-\yd\| )^2+\alpha \Rcal(v_n)\\
\leq (\rho_m+\|F(v_n)-F(\bu)-F'(\bu)(v_n-\bu)\|  +\|F'(\bu)(v_n-\bu)\|+\delta)^2+\alpha \Rcal(v_n)\\
\leq (\rho_m+c\lambda_n +\gamma_n+\delta)^2+\alpha \Rcal(v_n)\;.
\end{eqnarray*}
Denote
\begin{equation}\label{beta}
\beta_n:=(\rho_m+c\lambda_n +\gamma_n+\delta)^2.
\end{equation}
We use (\ref{breg_condition}) and get
\begin{eqnarray*}
\|F_m(\umn)-\yd\| ^2+\alpha D_\Rcal(\umn,\bu)&\leq & \beta_n+\alpha  \Rcal(v_n)-\alpha \Rcal(\bu)-\alpha\langle \xi,\umn-\bu\rangle\\
&=& \beta_n+\alpha D_\Rcal(v_n,\bu)-\alpha\langle \xi,\umn-v_n\rangle\\
&=& \beta_n+\alpha \lambda_n -\alpha\langle \omega,F'(\bu)(\umn-v_n)\rangle\\
&=& \beta_n+\alpha \lambda_n -\alpha\langle \omega,F'(\bu)(\umn-\bu)\rangle\\
&&+\alpha\langle \omega,F'(\bu)(v_n-\bu)\rangle\\
&\leq & \beta_n+\alpha \lambda_n+\alpha c\|\omega\| D_\Rcal(\umn,\bu)\\
&&+\alpha\|\omega\|\|F(\umn)-F(\bu)\| +\alpha\|\omega\|\gamma_n.
\end{eqnarray*}
Therefore,
\begin{equation}\label{interm}
\|F_m(\umn)-\yd\| ^2+\alpha(1-c\|\omega\|) D_\Rcal(\umn,\bu)\leq  \beta_n+\alpha \lambda_n+\alpha\|\omega\|(\zeta_n+\gamma_n),
\end{equation}
where $\zeta_n=\|F_m(\umn)-y\| $. Due to  (\ref{ineq1}), the term
$\alpha(1-c\|\omega\|) D_\Rcal(\umn,\bu)$ is non-negative.
Therefore,
\begin{equation}\label{ineq2}
\|F_m(\umn)-\yd\| ^2\leq \beta_n+\alpha \lambda_n+\alpha\|\omega\|(\zeta_n+\gamma_n).
\end{equation}
Using (\ref{noise})  we have
\begin{equation*}
\zeta_n^2\leq \left(\|F_m(\umn)-\yd\| +\|\yd-y\| \right)^2 \leq 2\|F_m(\umn)-\yd\| ^2+2\delta^2.
\end{equation*}
This together with inequality (\ref{ineq2}) implies
\[
\zeta_n^2\leq 2\beta_n+2\alpha \lambda_n+2\alpha\|\omega\|\zeta_n
+2\alpha\|\omega\|\gamma_n+2{\delta}^2,
\]
which yields
\begin{equation}\label{error_images}
\zeta_n\leq 2 \alpha\|\omega\|+\left(2{\alpha}^2\|\omega\|^2+  2{\delta}^2+ 2\beta_n+2\alpha \lambda_n
+2\alpha\|\omega\|\gamma_n\right)^{1/2}.
\end{equation}
From (\ref{interm}), it follows that
\[
\alpha(1-c\|\omega\|) D_\Rcal(\umn,\bu)\leq  \beta_n+\alpha \lambda_n+\alpha\|\omega\|\zeta_n+\alpha\|\omega\|\gamma_n,
\]
with $\zeta_n$ estimated above.
Using (\ref{beta}) and taking $\alpha\sim\max\{\delta,\lambda_n,\gamma_n\}$ yield the above convergence rate.
\end{proof}

\section{The case of separable Banach spaces}\label{separable}

Let $X$ be a separable Banach space. Consider a nested sequence of  finite dimensional subspaces $X_n, \,n\in\mathbb{N}$, such that
\[
{\overline{\cup_{n\in\mathbb{N}} X_n}}=X,
\]
where the closure is considered with respect to the norm topology of $X$. By letting  $Z=X$, a weak topology $\tau$ on $X$, a regularization functional $\Rcal:X\to [0,+\infty]$ and using the assumptions employed for the results in Section \ref{analysis}, one obtains stability and convergence  results similar to Proposition \ref{pr3.1} and Theorem \ref{th_conv}. Moreover, if $\Rcal$ is convex, then convergence rates can also be established.
\begin{remark} In some situations, convergence of a sequence $\{u_k\}$ to $u$ with respect to the topology $\tau$ and such that $\Rcal(u_k)\to\Rcal(u)$ implies $\|u_k-u\|\to 0$, as $k\to\infty$.
This is the case of locally uniformly convex reflexive Banach spaces  when $\tau$ is chosen as the weak topology on $X$ and $\Rcal=\|\cdot\|^p$, with $p\in (1,+\infty)$ such as Hilbert spaces, $L^p$ spaces, $W^{m,p}$ spaces, but also in $L^1$ when $\tau$ is the weak topology and $\Rcal$ is the Shannon entropy (see \cite{BorLew91}).
\end{remark}
\medskip

\textbf{Sparsity regularization}

Let $\{\phi_i\}$ be an orthonormal basis of $L^2(\Omega)$.  Denote by $X$ the Banach space $\ell^2$ which is identified with the
functions with bounded $\ell^2$ Fourier coefficients. Let
$X_n$ be the linear span of the first $n$ Fourier modes.

For sparsity regularization one usually takes ${\cal R}(u) = \sum_{i} |u_i|$, where
$u=\sum_i u_i \phi_i$. Thus, we consider the regularization
method of minimizing the functional
\begin{equation*}
u \to \|F(u)-\yd\|^2 + \alpha \sum_i |u_i|\,,
\end{equation*}

In this case, the topology $\tau$ on $X$ is the weak topology of $\ell^2$.
The regularization results apply also in this setting.

Another example of regularization term which promotes sparsity is
\begin{equation}\label{quasinorm}
 \Rcal(u)=\sum_i |u_i|^p,\,\,\,\,\,p\in (0,1).
\end{equation}
This setting with $X=\ell^2$ and $\tau$ taken as the weak topology of $\ell^2$ is also covered by the regularization  theory analyzed in this work. Moreover, convergence of a sequence $\{u_k\}$ to $u$ with respect to the topology $\tau$ and such that $\Rcal(u_k)\to\Rcal(u)$ implies convergence of  $\{u_k\}$ to $u$ relative to the quasinorm (\ref{quasinorm}), as $k\to\infty$, cf. \cite{GraHalSch08}.

\section{Particular nonseparable spaces}\label{partic_nonsep}
\subsection{The bounded variation functions space}

Recall that, for a bounded Lipschitz domain $\Omega\subset \R^N$ and for a given $N \in \N$, the space
$BV(\Omega)$ of $L^1(\Omega)$-functions of bounded variation mapping $\Omega$ into $\R$ can be defined as the set of functions
$w \in L^1(\Omega)$ such that the total variation of $w$ is finite, that is,
\begin{equation*}
 \int_{\Omega}\abs{Dw}_p = \sup \set{ \int w(x)\psi(x) dx:
\psi \in \Ccal_c^1(\Omega), \abs{\psi(x)}_{p'}\leq 1 \text{ for all } x\in \Omega}\;<\infty.
\end{equation*}
Here, $\abs{\cdot}_{p'}$ denotes the $l_{p'}$ vector norm where $p'=p/(p-1)$ is the conjugate exponent to $p$.
In particular we are interested in the cases $p=1,2$, where $l_{p'}=l_\infty, l_2$. The case  $p=2$ corresponds to
\emph{isotropic} total variation.

Let us recall several properties of the space $BV(\Omega)$:
\begin{itemize}
 \item It is the dual of a separable Banach space (see \cite[Remark 3.12]{AmbFusPal00}) when provided with the norm
\[
\|u\|_{BV}=\|u\|_{L^1}+ \int_{\Omega}\abs{Du}_p.
\]
\item  The space $BV(\Omega)$  is continuously embedded  in $L^r(\Omega)$, where $1\leq r\leq \frac{N}{N-1}$.
\end{itemize}

We consider the setting $X=BV(\Omega)$, with $\tau$ being the
weak$^*$ topology on $BV(\Omega)$, and $Z=L^1(\Omega)$.

The functional $\cal{R}$ is the total variation seminorm.
Consider
\begin{equation*}
 d(u,v) = \|u-v\|_{L^1(\Omega)}+\left|\int_\Omega |Du|_p - \int_\Omega |Dv|_p\right|\;.
\end{equation*}
The metric $d$ is  the metric used also in \cite{BelLus03},  which
gives the so-called \textit{strict convergence}, according to
\cite{AmbFusPal00}. A similar idea is developed in \cite{Leo05, Leo07}. Moreover, the strict convergence of a sequence $\{u_k\}$ to $u$ is equivalent
to convergence  with respect to the topology $\tau$ together with $\int_\Omega |Du_k|_1 \to \int_\Omega |Du|_1 $ as $k\to\infty$,
since weak$^*$  convergence of a sequence $\{u_k\}$ to $u$ in
$BV(\Omega)$  is equivalent to boundedness of $\{\|u_k\|_{BV})\}$
together with convergence of $\{u_k\}$ to $u$ in $L^1(\Omega)$ -
see, e.g., Proposition 3.13 in  \cite{AmbFusPal00}.

The choice of the vector norm in the definition of the bounded variation seminorm is of special importance
for approximating the $BV$space by subspaces consisting of piecewise constant functions.
Assume that $\set{\Omega_j}$ is a decomposition of $\Omega$,
and consider the following finite dimensional spaces:
\begin{equation*}
  X_n=\set{u_n =\sum_{j=1}^n u^j\chi_{\Omega_j}: u^j \in \R, 1\leq j\leq n}\;.
\end{equation*}
When considering a partition of $\Omega$ into uniform
parallelepipeds we can only guarantee the density assumption
\eqref{closure} when considering the $l_1$-norm in the definition of
$BV$ - see \cite{CasKunPol99}. If one wants an isotropic behavior of the
regularization term one has to consider the $l_2$ norm in the
definition of the $BV$-seminorm. The problem is that in the case of
uniform parallelepipeds (for instance pixels in imaging), there is
no convergence with respect to this isotropic $BV$-seminorm. One has
to consider a more general partition of the domain $\Omega$, that
allows to approximate level lines with any direction. One idea is to
use an irregular triangulation. These observations have been made by
\cite{BelLus03} and \cite{CasKunPol99}, here we only state their
main result, concerning anisotropic and isotropic total variation:

\begin{theorem}[\cite{BelLus03,CasKunPol99}]\label{th:BV}
Let $\Omega\subset \R^n$ be a polygonal domain and let $h_0>0$.
\begin{itemize}
\item
Given $u\in BV(\Omega)$, then there exists a family $\set{t_h: 0<h\leq h_0}$ of \textbf{triangulations} of
$\Omega$ such that the mesh-size of $t_h$ is at most $h$, and functions $u_h\in \mathcal{A}_h^0$, where $\mathcal{A}_h^0$
denotes the space of piecewise constant functions corresponding to the triangulation $t_h$, such that
\begin{equation*}
 \norm{u-u_h}{L^1} + \abs{\int_\Omega \abs{Du}_2 - \int_\Omega \abs{Du_h}_2 }\rightarrow 0\quad
 \text{as }h\rightarrow 0\;.
\end{equation*}
\item Given $u\in BV(\Omega)$. Let $\set{\Omega_j}$ be a decomposition of $\Omega$, into \textbf{parallelepipeds}, such that the
maximal length of a parallelepiped is smaller then $h$.
Then there exist functions $u_h\in V_h^0$, where $V_h^0$ denotes the space of piecewise constant functions corresponding to the
partition $\set{\Omega_i}$, such that
\begin{equation*}
 \norm{u-u_h}{L^1} + \abs{\int_\Omega \abs{Du}_1 - \int_\Omega \abs{Du_h}_1 }\rightarrow 0\quad
 \text{as }h\rightarrow 0\;.
\end{equation*}
\end{itemize}
\end{theorem}

Another type  of approximation of $X$  is the '$J_{\mu}$-approximation property'
employed in \cite{FitKee97} which, adapted  to our notation reads as follows:

$X_n$ is a $\Phi_{\alpha}$-approximation of $X$ if, for each $u\in
X$, there exists a sequence  $\{u_n\}\subset X_n$ such that
$\|u-u_n\|_Z\rightarrow 0$ and $\Phi_\alpha(u_n)\rightarrow
\Phi_\alpha(u)$ as $n\rightarrow\infty$, for any $\alpha\geq 0$,
where $\Phi_\alpha$ is given by
\begin{equation*}
\Phi_\alpha(u)=\norm{F(u)-\yd} { }^2+\alpha \Rcal(u)\;.
\end{equation*}
In \cite{FitKee97} one
aims at approximating minimizers of $\Phi_\alpha$ (which depends on
$\alpha$) in $X$ by minimizers of $\Phi_\alpha$ in $X_n$ for a fixed
$\alpha>0$ , while our aim is to  approximate solutions of the
operator equation  by minimizers of $\Phi_\alpha$ in $X_n$ when  the
regularization parameter $\alpha$ depends  on the dimension $n$.

We summarize the results of this section in an example
\begin{example}
Let $X = BV(\Omega)$ with the weak* topology $\tau$. Moreover, let
$Z=L^1(\Omega)$ and let $d$ be the metric
\begin{equation}\label{metric_bv}
 d(u,v) = \|u-v\|_{L^1(\Omega)}+\left|\int_\Omega |Du| - \int_\Omega |Dv|\right|\;.
\end{equation}
Let $F_m$ and $F: \df \subseteq BV(\Omega) \to
L^2(\hat{\Omega})$
satisfy the related conditions in Assumption \ref{newassump}.

Then according to Theorem \ref{th:BV}, for every $u \in BV(\Omega)$
there exists  an approximating sequence of piecewise constant
functions. Consequently, minimization of the discretized regularized
problem is well--posed, stable, and convergent (cf. Proposition
\ref{pr3.1}). The piecewise constant regularizers $\{umn\}$
approximate the ${\cal R}$-minimizing solution $\bu$ in the sense of
the metric (\ref{metric_bv}).
\end{example}


\subsection{The bounded deformation functions space}

In the following let $\Omega = (0,1)^N$ the open unit cube.
We choose the simple geometry not to be forced to take into
account approximations of $\Omega$, or irregular meshes, for
the finite element method considered below.

The space $BD(\Omega)$ \cite{TemStr80} of functions with \textbf{bounded deformation}
in an open set $\Omega\subset \R^N$ is defined as the set of functions
$\ufett=(u^1,\dots, u^N) \in L^1(\Omega;\R^N)$ such that the symmetric distributional derivative
\begin{equation*}
 E_{ij}\ufett:=\frac{1}{2} (D_i u^j+D_ju^i)
\end{equation*}
is a (matrix-valued) measure with finite total variation in $\Omega$:
\begin{equation*}
 BD(\Omega):=\set{\ufett \in L^1(\Omega;\R^N), E_{ij}(\ufett) \in M_1(\Omega),i,j=1,\dots,N },
\end{equation*}
where $M_1(\Omega)$ denotes the space of bounded measures.
$BD(\Omega)$ is a nonseparable Banach space provided with the norm
\begin{equation*}
\norm{\ufett }{BD}=\norm{ \ufett }{L^1(\Omega;\R^N)}+
\underbrace{ \sum_{i,j}^N \int_\Omega \abs{E_{ij}(\ufett)}}_{=: \int_\Omega \abs{E \ufett}}\;.
\end{equation*}

This space is strictly larger than the space of bounded variation
functions $BV(\Omega;\R^N)$. It was introduced in \cite{Suq78} and
\cite{MatStrChr79} and has been widely considered in the literature
(see \cite{Tem83,TemStr80}) in connection with the mathematical
theory of plasticity. Several  interesting properties of the $BD$
space are as follows: It is the dual of a separable Banach space
(see \cite{TemStr80}); the space $BD(\Omega)$  is continuously
embedded  in $L^p(\Omega;\R^N)$, where $1\leq p\leq \frac{N}{N-1}$.
In addition the space $BD(\Omega;\R^N)$ is not separable -- if it
would be, the space $BV(\Omega;\R^N)$,
which is the subspace of
$BD(\Omega;\R^N)$ where all components $u^j$, $j=2,\ldots,N$ vanish,
would be as well. However, this is not true as stated already in the
previous example.

Let us return to the finite dimensional regularization framework
we investigate in this work.

Consider the setting $X=BD(\Omega)$, $\tau$ the weak$^*$ topology on
$BD(\Omega)$ and $Z=L^1(\Omega;\R^N)$.
We associate with $n\in \N$ the discretization size $h_n:=\frac{1}{n}$,
use multi-indices $\av:=(\alpha_1,\dots,\alpha_N)$, and set
$A:=\set{0,\dots,n}^N$.

We consider the finite-dimensional product space of piecewise linear splines
\begin{equation*}
 X_n:=\set{\ufett_n=\sum_{\alpha \in A}\sum_{k=1}^N u^k_\av \Deltaxih : u_\av^k \in \R}
\end{equation*}
where $\xifett_\av \in  h_n \set{0,1,\ldots,n}^N$,
and $\Delta$ is the following function
\begin{equation*}
 \Delta(\boldsymbol{x}):=\prod_{i=1}^N  \max(0,1-\abs{x^i} )\;.
\end{equation*}
This finite element discretization has already been used in \cite{Cha04}
for numerical minimization of variational energies.
For $\boldsymbol{x} \in (0,h_n)^N$ the derivative of
$\Delta( (\boldsymbol{x}-\xifett_\av )/h_n)$ in direction $\boldsymbol{e}_j$ is given by
\begin{equation*}
 D_j\Delta\left( \frac{\boldsymbol{x}-\boldsymbol{\xi}_\av}{h_n}
\right)=
 \signum(\xifett_\av^j-x^j)\frac{1}{h_n} \prod_{i\not=j}^N
 \left( \max\left(0,1-\frac{\abs{x^i-(\xifett_\av)^i}}{h_n} \right) \right)\;.
\end{equation*}
For every $\av$ with $\alpha_i<n$ and for every $k, 1\leq k\leq N$
define  $A_\av^k:=\set{\beta \in A_\av , \bve_k=\av_k}$.
Additionally define $\Omega_\alpha$ as the $N-$dimensional
cube, spanned by the vectors $\set{\xifett_\av+\boldsymbol{e}_k h_n}_{k=1 \dots N}$, and
$A_\av=\bigcup_{k=1}^N A_\av^k$ (see Figure \ref{fi:alphasets}).

\begin{figure}[h]
\begin{tikzpicture}
 \fill[orange!20!] (-1cm,-1cm) rectangle (1cm,1cm);
\fill (-1 cm,-1 cm) circle (1pt) node[below] {$\xifett_{\av}$};
\fill (1 cm,-1 cm) circle (1pt) node[below] {$\xifett_{\av+(1,0)}$};
\fill (-1 cm,1 cm) circle (1pt) node[above] {$\xifett_{\av+(0,1)}$};
\fill (1 cm,1 cm) circle (1pt) node[above] {$\xifett_{\av+(1,1)}$};
\node at (0cm,0cm) {$\Omega_\alpha$};
\node[right] at (3cm, 0.6cm) {$A_\av=\set{\av,\av + (0,1),\av + (1,0),\av + (1,1)}$};
\node[right] at (3cm, -0.0cm){$A_\av^1=\set{\av, \av + (0,1)}$};
\node[right] at (3cm, -0.6cm){$A_\av^2=\set{\av, \av + (1,0)}$};
\end{tikzpicture}
\caption{an example for the sets $A_\av$}
\label{fi:alphasets}
\end{figure}
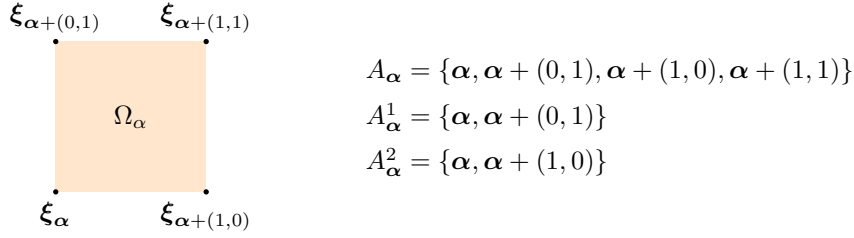
Moreover,
for $\beta \in A_\alpha$
we have
\begin{align}
\int_{\Omega_{\alpha}}
D_j \Delta
\left(\frac{\boldsymbol{x}-\xifett_\beta}{h_n}\right)
d\boldsymbol{x}
&=\begin{cases}
- \left( \frac{h_n}{2} \right)^{N-1} &\text{if }\beta^j=\alpha^j\\
+ \left( \frac{h_n}{2} \right)^{N-1} &\text{if }\beta^j=\alpha^j+1    \;
\end{cases}
\label{eq:sum}
\end{align}

In the following we prove the main result on a pseudometric.
\begin{theorem}\label{th:mainBD}
We assume that $h_n\rightarrow 0$ when $n\rightarrow \infty$. Then for
every $\ufett\in BD(\Omega,\R^N)\cap L^r(\Omega,\R^N),1\leq r < \infty$,
we can find a sequence $\set{\ufett_n}$, with $\ufett_n \in X_n$, such that
\begin{equation*}
 \lim\int_\Omega \abs{\ufett-\ufett_n}^r dx=0 \text{ and }
\lim_{n\rightarrow \infty} \int_\Omega \abs{E \ufett_n}= \int_\Omega \abs{E \ufett}\;.
\end{equation*}
Setting
\begin{equation}\label{eq:dBD}
d( {\ufett_1},{\ufett_2})=\norm{ \ufett_1-\ufett_2 }{L^1(\Omega;\R^N)}+
\abs{\int_\Omega \abs{E \ufett_1}\;-\int_\Omega \abs{E \ufett_2}\;},
\end{equation}
we obtain
$\lim_{n\rightarrow \infty} d(\ufett_n,\ufett)=0$.
\end{theorem}
In order to prove this theorem, we need some additional facts on $BD$,
given by the following  Lemmas.

\begin{lemma}\label{lem:discrete}
For every $n\in \N$ the inclusion $X_n\subset BD(\Omega,\R^N)$ holds and for each
$\ufett=(u^1,\dots,u^N)\in X_n$,
\begin{eqnarray*}
  \int_\Omega\abs{E\ufett}&=&
\sum_{\alpha}
\sum_{k=1}^N
\abs{\sum_{\bve \in A_{\alpha}^k}  \left(u^k_{\bve+\boldsymbol{e}_k}- u^k_\bve \right) }
\left(\frac{h_n}{2}\right)^{N-1} \\
&&+
\sum_{\alpha}
\sum_{k\not=l}^N
\frac{1}{2}\abs{
\sum_{\bve \in A_{\alpha}^l}  \left(u^k_{\bve+\boldsymbol{e}_l}- u^k_\bve \right)+
\sum_{\bve \in A_{\alpha}^k}  \left(u^l_{\bve+\boldsymbol{e}_k}- u^l_\bve \right)
}\left(\frac{h_n}{2}\right)^{N-1}
\end{eqnarray*}
\end{lemma}

\begin{proof}
From the definition of $\ufett \in X_m$ we obtain
\begin{align*}
 D_l u^k(\boldsymbol{x}) &=
\sum_{\av} u^k_\av D_l \Delta \left( \frac{\boldsymbol{x}- \xifett_\av}{h_n} \right)\;.
\end{align*}
Moreover, since $\Omega=\bigcup_{\av: \alpha_i<n} \Omega_\av$,
we can split up the integral. When integrating $D^lu_k+D^ku_l$ over $\Omega_\alpha$,
we only have to sum over those $\bve$
for which $\Delta\left(\frac{\boldsymbol{x}-\xifett_{ \bve}}{h_n} \right)\not=0$.
Hence we only sum over $\bve \in A_\alpha$ in the inner sum.
\begin{align*}
\int_{\Omega} \abs{D_l u^k + D_k u^l}
&=\sum_\av \int_{\Omega_\av} \abs{D_l u^k + D_k u^l}\\
&=\sum_\av \int_{\Omega_\alpha}
\abs{ \sum_{\bve \in A_\alpha}
\left(
  u^k_\bve D_l \Delta\left( \frac{\boldsymbol{x}- \xifett_\bve}{h_n} \right)+
  u^l_\bve D_k \Delta\left( \frac{\boldsymbol{x}- \xifett_\bve}{h_n} \right)
\right)
} \;.
\end{align*}
Next we use \eqref{eq:sum} and obtain
\begin{align*}
\int_{\Omega_\alpha}
\abs{D_l u^k + D_k u^l } =
\abs{
\sum_{\bve \in A_{\alpha}^l}  \left(u^k_{\bve+\boldsymbol{e}_l}- u^k_\bve\right)+
\sum_{\bve \in A_{\alpha}^k}  \left(u^l_{\bve+\boldsymbol{e}_k}- u^l_\bve\right)
}\left(\frac{h_n}{2}\right)^{N-1}
\end{align*}

The lemma follows from summing over all $\alpha$ and all $l,k=1,\dots,N$.
\end{proof}

\begin{lemma}\label{le:Eijproperties}
 \begin{enumerate}
    \item If $\set{\ufett_n}\subset BD(\Omega)$ and $\ufett_n\rightarrow \ufett$ in $(L^1(\Omega))^N$, then
      \begin{equation*}
      \int_\Omega \abs{E_{ij}\ufett}\leq  \liminf_{n\rightarrow \infty} \int_\Omega \abs{E_{ij} \ufett_n}\;.
      \end{equation*}
    \item For every $\ufett\in BD(\Omega)\cap L^r(\Omega), 1\leq r<\infty$,
      there exists a sequence $\set{u_n}\subset \Ccal^\infty(\overline{\Omega})$
      such that
      \begin{align*}
        \lim_{n\rightarrow \infty}\int \abs{\ufett-\ufett_n}^r dx&=0 &
        \lim_{n\rightarrow \infty}\int_\Omega  \abs{E_{ij}\ufett_n}&=\int_\Omega \abs{E_{ij}\ufett}\;.
       \end{align*}
 \end{enumerate}
\end{lemma}

\begin{proof}
 \begin{enumerate}
  \item Follows from standard properties of convex measures.
  \item See \cite[Theoreme 3.2, Chapitre II]{Tem83}.
\end{enumerate}
\end{proof}

Now we are ready for the proof of Theorem \ref{th:mainBD}.

\begin{proof} [Theorem \ref{th:mainBD}]
The $L^p$-convergence of the picewise polinomial functions $\ufett_n$ can be found in 
 the book of Ciarlet \cite{Cia02}.
Due to Lemma \ref{le:Eijproperties}, we can assume that $\ufett \in \mathcal{C}^\infty (\overline \Omega)$.
Set the coefficient of $\ufett_n$ as $(u_n)^k_\av:=u^k(\xifett_\av)$, then we have
\begin{equation*}
\ufett_n=\sum_{k=1}^N \sum_{\av}  u^k(\xifett_\av) \Deltaxih\;.
\end{equation*}
From Lemma \ref{lem:discrete} it follows that
\begin{equation}\label{eq:Ekl}
\begin{aligned}
\int_{\Omega_\alpha} \abs{E_{kl} \ufett_n}
= & \int_{\Omega_\alpha} \abs{D_l u^k_n + D_k u^l_n } \\
= & \abs{
\sum_{\bve \in A_{\alpha}^l}  \left(u^k(\xifett_{\bve+\boldsymbol{e}_l})- u^k(\xifett_\bve)\right)+
\sum_{\bve \in A_{\alpha}^k}  \left(u^l(\xifett_{\bve+\boldsymbol{e}_k})- u^l(\xifett_\bve) \right)
}\left(\frac{h_n}{2}\right)^{N-1}
\end{aligned}
\end{equation}
Next we use the mean value theorem: 
For $m,n \in \set{l,k}, m\not= n$ we can find points $\etafett_{\bve,m,n}$ between
$\xifett_{\bve}$ and $\xifett_{\bve+\boldsymbol{e}_m}$ such that
\begin{align*}
u^k(\xifett_{\bve+\boldsymbol{e}_l})- u^k(\xifett_\bve) &= h_n
D_l u^k(\etafett_{\bve,k,l})\;,\\
u^l(\xifett_{\bve+\boldsymbol{e}_k})- u^l(\xifett_\bve) &= h_n
D_k u^l(\etafett_{\bve,l,k})\;,
\end{align*}
hence from \eqref{eq:Ekl} it follows that
\begin{align*}
 \int_{\Omega_\alpha} \abs{E_{kl} \ufett_n}&= \abs{
\sum_{\bve \in A_{\alpha}^l}    D_l u^k(\etafett_{\bve,k,l})+
\sum_{\bve \in A_{\alpha}^k}   D_k u^l(\etafett_{\bve,l,k})
}\underbrace{h_n^N}_{\abs{\Omega_\av}} \left(\frac{1}{2}\right)^{N-1}
\end{align*}
Summing over all $\alpha$ and $l,k=1,\dots N$ and taking the limit
$h_n\rightarrow 0$ we have that
\begin{equation*}
\lim_{n\rightarrow \infty} \int_\Omega \abs{E_{ij} \ufett_n} =
\int_\Omega \abs{E_{ij} \ufett } \;.
\end{equation*}
\end{proof}

We summarize the results of this section as follows:

\begin{example}\label{example_bd}
Let $X =BD (\Omega)$ and $\tau$ the weak* topology.
Moreover, let $Z=L^1(\Omega,\R^N)$ and let
$d$ be the metric
\begin{equation}\label{metric_BD}
 d(\ufett,\boldsymbol{v}) = \|\ufett-\boldsymbol{v}\|_{L^1(\Omega,\R^N)}
 + \left|\     \int_\Omega \abs{E\ufett} - \int_\Omega \abs{E \boldsymbol{v}}  \right|\;.
\end{equation}
Let $F_m$ and $F: \df \subseteq BV(\Omega) \to L^2(\hat{\Omega})$  satisfy the related conditions in Assumption \ref{newassump}.

Take $h=1/n$. Then according to Theorem \ref{th:mainBD}, for every 
$\ufett \in BD(\Omega,\R^N)\cap L^r(\Omega,\R^N)$, $1\leq r<\infty$
there exists an approximating sequence
of piecewise constant functions in the sense of metric $d$.
Consequently, minimization of the discretized regularized problem is well--posed, stable,
and convergent according to Proposition \ref{pr3.1}. 
The piecewise constant regularizers $\{ \ufett_{m,n}^{\alpha,\delta}  \}$ approximate the ${\cal R}$-minimizing
solution $\overline \ufett $ in the sense $\lim \ufett_{m,n}^{\alpha,\delta}=\overline \ufett$ in the weak star topology and  
$\lim  \int_\Omega \abs{E\ufett_{m,n}^{\alpha,\delta}}  =\int_\Omega \abs{ E \overline \ufett}$.
\end{example}


\subsection{The $L^\infty$ space}

In this section we analyze the following regularization method
\begin{equation*}
 u \to \|F(u)-\yd\|^2 + \alpha \|u\|_{\infty}\;.
\end{equation*}

In the sequel we show that there exist finite dimensional subspaces
$\{X_n\}$ of $L^{\infty}(\Omega)$ satisfying equality
(\ref{closure}). More precisely, there exist finite dimensional
subspaces $\{X_n\}$ of $L^{\infty}(\Omega)$ such that for any $u\in
L^{\infty}(\Omega)$ one can find $u_n\in X_n$, $n\in\mathbb{N}$ with
\begin{equation}\label{approx_inf}
\lim_{n\rightarrow \infty}\left(\norm{u_n-u}{L^p} + |\|u_n\|_{\infty}-\|u\|_{\infty}|\right)=0, \quad p\in [1,+\infty).
\end{equation}

It is known that $L^{\infty}(\Omega)$ is not separable, while
$L^p(\Omega)$, for every $p\in[1,+\infty)$, is separable. However,
every function in $L^{\infty}(\Omega)$ can be approximated uniformly
and thus, in the $L^{\infty}(\Omega)$ - norm by a sequence of simple
functions. The proof of this statement is constructive. However, it
provides a nonlinear approximation, in the sense that the piecewise
constant functions which approximate the $L^{\infty}(\Omega)$
function do not yield linear subspaces - see, e.g., \cite[Section
3.2]{DeV98}. This does not fit the theoretical framework we consider
here. An alternative is to consider approximations of
$L^{\infty}(\Omega)$ functions by piecewise constant functions in a
weaker topology, as shown in the sequel.

For the sake of simplicity let $\Omega =(0,1)^N$.
Assume that $\set{\Omega_j}$ is a decomposition of $\Omega$  in
parallelepipeds with maximal diagonal length $h_n$ as in  \cite{CasKunPol99}, and consider the
following finite dimensional subspaces of $L^{\infty}(\Omega)$:
\begin{equation*}
  X_n=\set{u_n =\sum_{j=1}^n u^j\chi_{\Omega_j}: u^j \in \R, 1\leq j\leq n}\;.
\end{equation*}

The following results are essential in proving the main statement of this section:

First denote by $J_{\epsilon}\ast u$ the mollification of $u$, for every $\epsilon>0$.

\begin{theorem}\label{mollif}[p. 36 \cite{Ada03}] Let $u$ be a function which is defined on ${\mathbb{R}}^N$ and vanishes identically outside $\Omega$.

If $u\in L^p(\Omega)$, then $J_{\epsilon}\ast u\in C^{\infty}({\mathbb{R}}^N)$, in fact $J_{\epsilon}\ast u\in C^{\infty}(\overline \Omega)$
and  $J_{\epsilon}\ast u\in L^p(\Omega)$, for every $p\in [1,+\infty)$. Also,
\begin{equation}\label{lp}
 \lim_{\epsilon\rightarrow 0_+}\|J_{\epsilon}\ast u-u\|_p=0,
\end{equation}
\begin{equation}\label{ineq}
 \|J_{\epsilon}\ast u\|_p\leq \|u\|_p, \quad \text{for all }\;\epsilon>0.
\end{equation}
\end{theorem}

\begin{lemma}\label{max} For every $n\in\mathbb{N}$, one has $X_n\subset L^{\infty}(\Omega) $  and, for each $u_n\in X_n$,
\begin{equation*}
 \|u_n\|_{\infty}=\max_{1\leq j\leq n} |u^j|.
\end{equation*}
\end{lemma}

\begin{proof}
 Follows immediately from the definition of $u_n$, taking into account that $\mu(\Omega_j)>0$, for $1\leq j\leq n$.
\end{proof}

\begin{lemma}\label{converge} If $\{u_n\}\subset L^{\infty}(\Omega)$ and $u\in L^{\infty}(\Omega)$
are such that $\norm{u-u_n}{L^p} \rightarrow 0$ as $n\rightarrow
\infty$, for some $p\in [1,+\infty)$ and
$\|u_n\|_{\infty}\leq\|u\|_{\infty}$ for every $n\in\mathbb{N}$,
then $\{u_n\}$ converges weakly$^*$ to $u$ and
$\|u_n\|_{\infty}\rightarrow\|u\|_{\infty}$ as $n\rightarrow
\infty$.
\end{lemma}

\begin{proof}
Since $\set{\|u_n\|_{\infty}}$ is bounded, there exists a subsequence $\{u_k\}$ which converges weakly$^*$ to
some $v\in L^{\infty}(\Omega)$, cf.  Alaoglu-Bourbaki Theorem, \cite[p. 70]{Hol75}.
By the definition of weak$^*$ convergence, one obtains that $\{u_k\}$ converges also weakly, with respect to $L^p(\Omega)$
to some $v$. Therefore $u=v$. In fact, every subsequence of $\{u_n\}$ converges weakly$^*$ to $u$, which yields weak$^*$ convergence
of the entire sequence $\{u_n\}$ to $u$. In addition, the weak$^*$ lower semi continuity of the $L^\infty$-norm
implies $\|u\|_{\infty}\leq\liminf_{n\rightarrow\infty}\|u_n\|_{\infty}$. Thus the assertions are proved.
\end{proof}

Note that the previous result can be relaxed by assuming only weak convergence of $\{u_n\}$ in $L^p(\Omega)$.

\begin{lemma}\label{density_cont} For every $u\in C^{\infty}(\bar \Omega)$, there exists a sequence $\{u_n\}$ with $u_n\in X_n$ such that
  \[
\lim_{n\rightarrow \infty}\left(\norm{u_n-u}{L^p} + |\|u_n\|_{\infty}-\|u\|_{\infty}|\right)=0,\quad p\in [1,+\infty).
\]
 \end{lemma}

\begin{proof}  Let $\xi_j$ be the gravity centers of the parallelepipeds $\Omega_j$.
 Define
\[
u_n=\sum_{j=1}^n u(\xi_j)\chi_{\Omega_j}.
\]
Then $\norm{u_n-u}{L^p}\rightarrow 0$ - see, for instance the book of Ciarlet \cite{Cia02}.
Moreover, by using Lemma \ref{max},
\[
\|u_n\|_{\infty}=\max_{1\leq j\leq n} |u(\xi_j)|\leq  \max_{x\in\bar\Omega} |u(x)|=\|u\|_{\infty},\quad \text{ for all }\; n\in\mathbb{N},
\]
where the last equality holds due to the continuity of $u$ on $\bar\Omega$. Thus, Lemma \ref{converge} applies and yields  $|\|u_n\|_{\infty}-\|u\|_{\infty}|\rightarrow 0$ as $n\rightarrow \infty$.
\end{proof}

\begin{theorem}\label{density_l}
Assume that $h_n\rightarrow 0$ when $n\rightarrow \infty$. Then for every $u\in L^{\infty}(\Omega) $  one can find    $u_n\in X_n$ such that  (\ref{approx_inf}) holds.
\end{theorem}

\begin{proof}
Let $u\in L^{\infty}(\Omega)$ and $p\in[1,+\infty)$.  By Theorem \ref{mollif}, there exists $\{u_j\}\subset C^{\infty}({\mathbb{R}}^N)$, in fact in $C^{\infty}(\bar\Omega)$, with $\{u_j\}\subset L^p(\Omega)$ such that
\[
\lim_{j\rightarrow \infty}\|u_j-u\|_p=0,\,\,\,\,\,\mbox{and}\,\,\,\,\,\|u_j\|_p\leq \|u\|_p,\quad \text{ for all }\; j\in\mathbb{N}.
\]
By letting $p\rightarrow+\infty$ in the last inequality and using Theorem 2.8 on p. 25 in \cite{Ada75},  one also has
\begin{equation*}
 \|u_j\|_{\infty}\leq \|u\|_\infty, \quad \text{ for all }\; j\in\mathbb{N}.
\end{equation*}
Lemma \ref{converge} applies and yields $\|u_j\|_{\infty}\rightarrow\|u\|_{\infty}$ as $j\rightarrow \infty$.

By consequence, every function in $L^{\infty}(\Omega)$ can be approximated  in the sense stated at (\ref{approx_inf})
by functions from $C^{\infty}(\overline \Omega)$. Since every function $v\in C^{\infty}(\bar \Omega)$ can be approximated by
$u_n\in X_n$ as in \eqref{approx_inf} due to Lemma \ref{density_cont}, the conclusion follows immediately.
\end{proof}

\begin{remark}
One can define  the subspaces $X_n$ in the previous theorem also by means of piecewise polynomial
functions of degree no bigger than one in each variable, which are continuous on $\bar \Omega$.
Moreover, one can employ $n$-simplices instead of parallelipipeds and piecewise linear functions,
according to Remark 3.8 in \cite{CasKunPol99}.
\end{remark}

We summarize the results of this section as follows:

\begin{example}\label{example_infty}
Let $X =L^{\infty} (\Omega)$ and $\tau$ the weak* topology.
Moreover, let $Z=L^p(\Omega)$, $p\in (1,+\infty)$ and let
$d$ be the metric
\begin{equation}\label{metric_inf}
 d(u,v) = \|u-v\|_{L^p(\Omega)}+|\|u\|_{\infty}-\|v\|_{\infty}|\;.
\end{equation}
Let $F_m$ and $F: \df \subseteq BV(\Omega) \to L^2(\hat{\Omega})$  satisfy the related conditions in Assumption \ref{newassump}.

Take $h=1/n$. Then according to Theorem \ref{density_l}, for every $u \in L^{\infty} (\Omega)$ there exists  an approximating sequence
of piecewise constant functions in the sense of metric $d$. Consequently, minimization of the discretized regularized problem is well--posed, stable,
and convergent according to Proposition \ref{pr3.1}. The piecewise constant regularizers $\{\umn\}$  approximate the ${\cal R}$-minimizing
solution $\bu$ in the sense $\lim \umn=\bu$ in the weak star topology and  $\lim\|\umn\|_{\infty}=\|\bu\|_{\infty}$.
\end{example}
\medskip
This convergence described in the previous sentence is weaker than
convergence with respect to the metric (\ref{metric_inf}). The fact
that these two types of convergence are not equivalent, by contrast
to the $BV$ and $BD$ cases with the two corresponding convergence
types,  is shown by the following counterexample.

Consider the Rademacher functions $f_n:[0,1]\rightarrow \{-1,1\}$,
\[
f_n(t)=(-1)^{i+1}\,\,\,\mbox{if}\,\,\,x\in\left[\frac{i-1}{2^n},\frac{i}{2^n}\right),\,\,\,1\leq i\leq 2^n.
\]
This sequence converges weakly
star to zero in  $L^\infty([0,1])$, but not in the $L^1([0,1])$ norm. Moreover, consider $g_n:[0,2]\rightarrow \mathbb{R}$,
\[
g_n(t)=f_n(t),\,\,\,\mbox{if}\,\,\,t\in[0,1]
\]
and $g_n(t)=1$ for $t\in[1,2]$.
Then $\{g_n\}$ converges weakly star
to $\chi_{[1,2]}$, the characteristic function
of $[1,2]$ in  $L^\infty([0,2])$ but not in the $L^1([0,2])$ norm.
Note that $\|g_n\|_{\infty}=\|\chi_{[1,2]}\|_{\infty}=1$. Thus, $\{g_n\}$ is a sequence in $L^\infty([0,2])$ which converges
weakly star to $\chi_{[1,2]}$ and such that $\lim_{n\rightarrow\infty}\|g_n\|_{\infty}=\|\chi_{[1,2]}\|_{\infty}$, but
\mbox{$\lim_{n\rightarrow\infty}\left(\norm{g_n-\chi_{[1,2]}}{L^1} + |\|g_n\|_{\infty}-\|\chi_{[1,2]}\|_{\infty}|\right)\neq 0$.}

\begin{remark} Given a direct problem formulated in $L^{\infty}$. Is it worth to formulate the
inverse problem in $L^{\infty}$ or is more appropriate to formulate the problem in $L^2$, where
we can get also $L^2$-approximations?
$L^2$ approximations are quite advantageous because in the $L^\infty$ case, one might not even get
convergence with respect to the $L^1$-norm, as demonstrated by the above counterexample.
However, it is sometimes desirable that the regularized solutions are guaranteed to belong also to the  $L^{\infty}$ space.
Moreover, the extremal behavior of the solution can be evaluated by
$L^{\infty}$ regularization, since convergence of the $L^{\infty}$-norm  of the regularized solutions
to the $L^{\infty}$ norm of the true solution is achieved. Also, in some inverse problems
high interest is given to estimating a linear functional of the solution rather than  the solution -
according to the mollifier idea in \cite{And86}, which corresponds to the weak-star approximation of our $L^\infty$ regularization results.
\end{remark}

\section{The inverse ground water filtration problem}\label{water}

We consider the problem of recovering the diffusion coefficient in
\begin{equation*}
 -(a u_x)_x=f\,\,\,\mbox{in}\,\,\,\Omega,
\end{equation*}
\begin{equation*}
u(0)=0=u(1),
\end{equation*}
with $f\in L^2[0,1]$. The operator $F$ is defined as the parameter-to-solution mapping
\[
F:\df:=\{a\in L^{\infty}[0,1]:a(x)\geq c>0, \,a.e.\}\to L^2[0,1],
\]
\[
a\mapsto F(a):=u(a),
\]
where $u(a)$ is the unique solution of the above equation and $c$ is a constant. More details about this ill posed problem can be found in  \cite{ItoKun94} and \cite[Chapter 1]{Cha09}. Note that $\df$ is a subset of the interior of the nonnegative cone of $L^{\infty}[0,1]$. In fact  $L^{\infty}$ is the  natural function space when formulating this problem. However, to the best of our knowledge, previous literature dealing with the problem from the regularization viewpoint has usually considered $\df$ in $H^1$ which is embedded in $L^{\infty}$, mainly due to the Hilbert space setting which is enforced by using $H^1$.  The operator $F$ is Fr\'echet differentiable from $L^{\infty}$ to $L^2$ - see \cite{ItoKun94}.
Note that $\df$  is closed and convex with respect to $L^2$, so it is weakly closed in $L^2$. This implies that $\df$ is weakly$^*$ closed in $L^{\infty}$.  Similarly one can argue that the operator $F$ is sequentially  weakly$^*$-weakly  closed.

 In the sequel we are going to employ approximation operators $F_m$ as in \cite{NeuSch90}.
Let $Y_m$ be the space of linear splines on a uniform grid of $m+1$ points in $[0,1]$, which vanish at $0$ and $1$. By using the variational formulation, let $u_m(a)\in Y_m$ be the unique solution of
\[
(a(u_m)_x,v_x)_{L^2}=(f,v)_{L^2}\,\,\,\mbox{for all}\,\,\,v\in Y_m.
\]
The operators $F_m$ are defined as
\[
F_m:\df:=\{a\in L^{\infty}[0,1]:a(x)\geq c>0, \,a.e.\}\to L^2[0,1],
\]
\[
a\mapsto F_m(a):=u_m(a),
\]
Then (\ref{closeness}) holds for $\rho_m=m^{-2}$  cf. \cite[Theorems 3.2.2, 3.2.5]{Cia02}, 
\[
\|F_m(a)-F(a)\|_{L^2}=\|u_m(a)-u(a)\|_{L^2}=O( \|a\|_{L^{\infty}} \cdot m^{-2}).
\]

Then by choosing the discretization of $L^{\infty}[0,1]$ as in Section \ref{partic_nonsep}, one obtains the convergence results of the
discretized regularization method as in Example \ref{example_infty}.

\section*{Acknowledgment}
The authors thank Prof. J.B. Cooper (Kepler University, Linz) for providing them the counterexample in Section \ref{partic_nonsep}. They are  grateful to S. Pereverzev (Radon Institute, Linz) for helpful discussions, and to A. Rieder (Karlsruhe University) and V. Vasin (Institute of Mathematics and Mechanics, Ekaterinburg) for  interesting references.
Also, they acknowledge the support by the Austrian Science Fund (FWF)
within the national research networks Industrial Geometry, project
9203-N12, and Photoacoustic Imaging in Biology and Medicine, project
S10505-N20 (C. P.  and O. S.) and by an Elise Richter
scholarship, project  V82-N18 (E. R.).


\end{document}